%% file: algebras-partial-functions-partI-revised-revised.tex
\documentclass{baustmsmodified}
\pdfoutput=1 

\citesort

\input{preamble}

\begin{document}
\runningtitle{Difference--restriction algebras of partial functions: axiomatisations and representations}
\title{Difference--restriction algebras of partial functions: axiomatisations and representations}
\author[1]{C\'elia Borlido$^{\,1}$}
\author[2]{Brett McLean$^{*, \,2}$}

\authorheadline{C\'elia Borlido and Brett McLean}

\support{ * Corresponding author.

  \noindent
  1.  CMUC, Departamento de Matem\'atica, Universidade de Coimbra,
  \mbox{3001-501}~Coimbra, Portugal
  \email{cborlido@mat.uc.pt},  \qquad
  ORCID: 0000-0002-0114-1572
  
  \noindent
  2. Laboratoire J. A. Dieudonn\'e UMR CNRS 7351, Universit\'e Nice
  Sophia Antipolis, 06108~Nice~Cedex~02,
  France\email{brett.mclean@unice.fr},  \qquad ORCID: 0000-0003-2368-8357}


\begin{abstract}
We investigate the representation and complete representation classes for algebras of partial functions with the signature of relative complement and domain restriction. We provide and prove the correctness of a finite equational axiomatisation for the class of algebras representable by partial functions. As a corollary, the same equations axiomatise the algebras representable by injective partial functions. For complete representations, we show that a representation is meet complete if and only if it is join complete. Then we show that the class of completely representable algebras is precisely the class of atomic and representable algebras. As a corollary, the same properties axiomatise the class of algebras completely representable by injective partial functions. The universal-existential-universal axiomatisation this yields for these complete representation classes is the simplest possible, in the sense that no existential-universal-existential axiomatisation exists.
\end{abstract}

\maketitle

\section{Introduction}
In J\'onsson and Tarski's seminal \cite{1951}, the authors produced
the very general definition of a \emph{Boolean algebra with operators}
by building upon the foundation provided by the class of Boolean
algebras. This factorisation of concerns into, firstly, the Boolean
order structure, and later, any additional operations is still
conspicuous when one examines the subsequently obtained duality
between Boolean algebras with operators and descriptive general frames
that has proved to be so important in modal logic \cite[Chapter 5:
Algebras and General Frames]{blackburn_rijke_venema_2001}. And a
similar remark can be made for the discrete duality that exists
between, on the one hand, complete and atomic Boolean algebras with
completely additive operators, and on the other hand, Kripke frames.

Whilst the Boolean framework is applicable to many famous classes of structures modelling relations---relation algebras and cylindric algebras are the foremost examples---it is not applicable to algebras modelling partial functions. This is for the simple reason that collections of partial functions cannot be relied upon to be closed under unions, only under unions of `compatible' functions. Though the theory of Boolean algebras with operators has been greatly generalised, weakening the ordered component all the way down to posets \cite{gehrke1994bounded, GEHRKE2001345, 10.2307/27588391}, it is not, in fact, a surfeit of order structure that is the culprit here. Indeed a moment's reflection reveals that the inclusion order is not in general enough to reveal whether two functions agree on any shared domain they may have.

There is growing interest in dualities for algebras modelling partial functions. More specifically, there are by now a number of duality theorems, proven by researchers working on inverse semigroups, having on one side: algebras modelling partial \emph{injective} functions and on the other: certain categories/groupoids or generalisations thereof \cite{lawson_2010, lawson2012non, kudryavtseva2017perspective, LAWSON201677, LAWSON2013117}.\begin{NoHyper}\footnote{There are also categorical equivalences, for example \cite{doi:10.1080/00927870902887096}, but they are not our target in this work.}\end{NoHyper} Recently, a similar result has been obtained for a specific signature of not-necessarily-injective partial functions \cite{2009.07895}, and there are also the dualities of \cite{Bauer_2013} and \cite{kudryavtseva2016boolean} relating to signatures not containing composition. Such classes of algebras arise naturally as inverse semigroups~\cite{wagnergeneralised}, pseudogroups~\cite{LAWSON2013117}, and skew lattices~\cite{Leech19967}, and within computer science appear in the theory of finite state transducers \cite{10.1145/2984450.2984453}, computable functions \cite{JACKSON2015259}, deterministic propositional dynamic logics \cite{DBLP:journals/ijac/JacksonS11}, and separation logic \cite{disjoint}. The dualities have been applied to classical areas of algebra including group theory~\cite{lawson_2010, lawson2012non,  LAWSON201677}, (linear) representation theory~\cite{lawson_margolis_steinberg_2013}, and the theory of $\mathrm C^*$-algebras~\cite{lawson_2010, lawson2012non, LAWSON201677}.

A natural question is: \emph{Can such duality results be organised into a general frame\-work in the spirit of Boolean algebras with operators?} In this paper we investigate algebras of partial functions in a signature that could be a candidate to be the equivalent of the Boolean signature in this framework. The signature is: \emph{relative complement} and \emph{domain restriction}. We have chosen the signature so as to provide us with two things:
\begin{itemize}
\item
a well-behaved order structure,
\item
compatibility information.
\end{itemize}
To be more specific, relative complement provides relativised Boolean
structure (see \Cref{c:3}), and domain restriction can characterise
compatibility via the equation $a \rest b = b \rest
a$. 

We show that the abstract class of isomorphs of such algebras of
partial functions is axiomatised by a finite number of equations
(\Cref{axiomatisation}). One might term this abstract class the class
of `compatibility algebras'. Our signature is relatively unusual in
that it does not include the composition operation on partial
functions.  But we envisage a future theory of `compatibility algebras
with operators', in which familiar operators on partial functions,
such as \emph{composition}, \emph{range}, and \emph{converse} may be
treated.

For the case of discrete dualities, one can make the following observation. First note that the duality between complete and atomic Boolean algebras and sets is a restricted case of a more general \emph{adjunction} between atomic Boolean algebras and sets. Then recall that  the atomic Boolean algebras are precisely the completely representable Boolean algebras. (\emph{Any} meets that exist become intersections, \emph{any} joins become unions.) Hence the adjunction can be viewed in more semantic terms as linking those algebras that are completely representable as fields of sets with the category of sets. 

With this observation in mind, we prepare the ground for a discrete duality for `compatibility algebras with operators' by identifying the completely representable algebras of our signature (\Cref{complete_axiomatisation}). It turns out that in this case, again all that is needed is to add the condition `atomic' to the conditions for representability. Such an outcome is not as automatic as it may seem, for there exist situations, for unary relations \cite{egrot}, for higher-order relations \cite{journals/jsyml/HirschH97a},  and for functions \cite{complete}, where complete representability is characterised by more complex properties.

In the sequel to this paper, \emph{Difference--restriction algebras of partial functions with operators: discrete duality and completion} \cite{diff-rest2}, we carry out one of these planned continuations of the project. There, we present an adjunction (restricting to a duality) for the category of completely representable algebras and complete homomorphisms, then extend to an adjunction/duality for completely representable algebras equipped with compatibility preserving completely additive operators.

\subsubsection*{Structure of paper} In \Cref{preliminaries}, we define formally the class of representable algebras that we wish to axiomatise, list a finite number of valid equations for these algebras, and begin to deduce some consequences of these equations. In \Cref{sec:domains}, we deduce further consequences, relating specifically to the semantic notion of domain inclusion. In \Cref{sec:filters}, we deduce some properties of filters.

 In \Cref{sec:rep}, we use a representation based on maximal filters to prove that our equations axiomatise both the algebras representable as partial functions (\Cref{axiomatisation}) and also the algebras representable as \emph{injective} partial functions (\Cref{cor:rep}).
 
  In \Cref{sec:complete}, we define formally the completely representable algebras and show that they are precisely the atomic representable algebras in both the partial function (\Cref{complete_axiomatisation}) and injective partial function (\Cref{complete_corollary}) cases.

\section{Basic definitions and properties}\label{preliminaries}

In this section, we start with the necessary definitions relating to
algebras of partial functions, present the set of equations that will
eventually become our first axiomatisation, and derive various
consequences of these equations.

Given an algebra $\algebra{A}$, when we write $a \in \algebra{A}$ or
say that $a$ is an element of $\algebra{A}$, we mean that $a$ is an
element of the domain of $\algebra{A}$. Similarly for the notation $S
\subseteq \algebra{A}$ or saying that $S$ is a subset of
$\algebra{A}$. We follow the convention that algebras are always
nonempty. If $S$ is a subset of the domain of a map $\theta$ then
$\theta[S]$ denotes the set $\{\theta(s) \mid s \in S\}$. Given an
binary operation $\bullet$ on $\cA$ and subsets $S_1, S_2 \subseteq
\cA$, we shall use $S_1 \bullet S_2$ to denote the set $\{s_1 \bullet
s_2 \mid s_1 \in S_1,\ s_2 \in S_2\}$.

We begin by making precise what is meant by partial functions and
algebras of partial functions. 
\begin{definition}
  Let $X$ and $Y$ be sets. A \defn{partial function} from $X$ to $Y$
  is a subset $f$ of $X \times Y$ validating
\begin{equation*}
  (x, y) \in f \text{ and } (x, z) \in f \implies y = z.
\end{equation*}
If $X = Y$ then $f$ is called simply a partial function on $X$.  Given
a partial function $f$ from $X$ to $Y$, its \defn{domain} is the set
\[\dom(f) \coloneqq \{x \in X \mid \exists \ y \in Y \colon (x, y) \in f\}.\]
\end{definition}

\begin{definition} 
  An \defn{algebra of partial functions} of the signature $\{-,
  \rest\}$ is a universal algebra $\algebra A = (A, -, \rest)$ where
  the elements of the universe $A$ are partial functions from some
  (common) set $X$ to some (common) set $Y$ and the interpretations of
  the symbols are given as follows:
\begin{itemize}

\item
The binary operation $-$ is \defn{relative complement}:
\[f - g \coloneqq \{(x, y) \in X \times Y \mid (x, y) \in f \text{ and
  } (x, y) \not\in g\}.\]

\item
The binary operation $\rest$ is \defn{domain restriction}.\begin{NoHyper}\footnote{This operation has historically been called \emph{restrictive multiplication}, where \emph{multiplication} is the historical term for \emph{composition}. But we do not wish to emphasise this operation as a form of composition.}\end{NoHyper} It is the restriction of the second argument to the domain of the first; that is:
\[ f \rest g \coloneqq \{(x, y) \in X \times Y \mid x \in \dom(f)
  \text{ and } (x, y) \in g\}\text{.}\]
\end{itemize}
\end{definition}
Note that in algebras of partial functions of the signature $\{-,
\rest\}$, the set-theoretic intersection of two elements $f$ and $g$ can be
expressed as $f - (f - g)$. We use the symbol~$\cdot$ for this derived
operation.

We also observe that, without loss of generality, we may assume $X =
Y$ (a common stipulation for algebras of partial functions). Indeed,
if $\cA$ is a $\{-, \rest\}$-algebra of partial functions from $X$ to
$Y$, then it is also a $\{-, \rest\}$-algebra of partial functions
from $X \cup Y$ to $X \cup Y$. In this case, this
non-uniquely-determined single set is called `the'
\defn{base}. However, certain properties may fail to be preserved by
changing the base. For instance, given sets $X$ and $X'$, if $f$ qualifies both as a partial function on $X$ and a partial function on $X'$, while it is true that $f$ is injective as a partial function on~$X$ if and only if it
is injective as a partial function on~$X'$, this is not the case for
surjectivity.

The collection of \emph{all} partial functions on some base $X$ is closed under relative complement and domain restriction, and thus gives an algebra of partial functions $\mathcal{PF}(X)$.

\begin{definition}
  An algebra $\algebra A$ of the signature $\{ -, \rest\}$ is
  \defn{representable} (by partial functions) if it is isomorphic to
  an algebra of partial functions, equivalently, if it is embeddable into $\mathcal{PF}(X)$ for some set $X$. Such an embedding of $\algebra A$ is a
  \defn{representation} of $\algebra A$ (as an algebra of partial
  functions).
\end{definition}

Just as for algebras of partial functions, for any $\{-, \rest\}$-algebra
$\cA$, we will consider the derived operation $\bmeet$ defined by
\begin{equation*}
  \label{complement}\tag{I}
  a \bmeet b := a - (a - b).
\end{equation*}

Algebras of partial functions of many other signatures have been investigated, and the corresponding representation classes axiomatised, often (but not always) finitely, and often (but not always) with equations. We will not enumerate all these results here, but for a treatment of some of the most expressive signatures to have been considered, see \cite{hirsch}. For a relatively comprehensive guide to this literature, see \cite[\S 3.2]{phdthesis}. 

 Focusing on signatures that, like ours, do not contain composition, first consider the signature $\{\rest, \pref\}$ (incomparable with ours), where $\pref$ is the operation known as \emph{preferential union} or alternatively as \emph{override}. Here, the representation class is precisely the right-handed strongly distributive skew lattices  \cite{LeechJonathan1992Nsl}, and thus finitely axiomatisable by equations.  See \cite{Bauer_2013} for the definition of right-handed strongly distributive skew lattices, where a duality theorem for this class is proven. 
 In \cite{JACKSON2021106532}, a finite equational axiomatisation is given for the signature---also incomparable with ours---of preferential union and \emph{update}. The paper \cite{BERENDSEN2010141} gives a finite equational axiomatisation for the signature $\{-, \pref\}$, which is more expressive than each of the three other signatures (ours and the two just mentioned).

In \Cref{sec:rep} we shall see (\Cref{axiomatisation}) that the
class of $\{-, \rest\}$-algebras that is representable by partial
functions is the variety axiomatised by the following set of
equations.
\begin{enumerate}[label = (Ax.\arabic*), leftmargin = *]
\item \label{schein1} $a - (b - a) = a$
\item \label{commutative} 
$a \bmeet b = b \bmeet a$
\item \label{schein3} $(a - b) - c = (a - c) - b$
\item \label{eq:8} $(a \rest c)\bmeet(b \rest c) = (a \rest b) \rest
  c$
\item \label{lifting}$ (a \bmeet b) \rest a = a \bmeet b$
\end{enumerate}
We call the $\{-, \rest\}$-algebras satisfying these axioms \defn{difference--restriction algebras}.

Algebras of the signature $\{-\}$ validating axioms \ref{schein1} --
\ref{schein3} are called \defn{subtraction algebras}
\cite{abbott1969sets}. It is known that these equations axiomatise the
$\{-\}$-algebras representable as an algebra of sets equipped with
relative complement (see, for example, \cite[Theorem~1 +
Example~(2)]{schein1992difference}). Hence \ref{schein1} --
\ref{schein3} are sound for $\{ -, \rest\}$-algebras of partial
functions and therefore for all representable $\{ -,
\rest\}$-algebras. We also know immediately that any (isomorphism
invariant) property of sets with~$-$ is a consequence of \ref{schein1}
-- \ref{schein3}.  One particular consequence that will be often used
in the rest of the paper without further mention is that the derived
operation $\bmeet$ provides the structure of a semilattice with bottom
$0 \coloneqq a - a$ (independent of choice of $a$). Similarly, we will
also use the fact that $0$ acts as a right identity for $-$ without
further remark. Three further properties that we will find useful are
the following.
\begin{align}
  \label{disjoint_0} b \bmeet (a - b) &= 0\\
  \label{eq:9} a - (a\bmeet b) & = a-b\\
  \label{eq:10} a\bmeet (b-c) &= (a\bmeet b)-c
\end{align}

We also observe that axioms~\ref{eq:8} and~\ref{lifting} are stated
without explicitly using the operation~$-$. It turns out that many
results in this paper do not depend on the algebraic properties
of~$-$, but only on the semilattice operation $\cdot$ it defines. For
that reason, we will use the name \defn{restriction semilattice} for
algebras over the signature $\{\cdot, \rest\}$ whose
$\{\cdot\}$-reduct is a semilattice and that satisfy axioms \ref{eq:8}
and \ref{lifting}. Note that, in general, a restriction semilattice
may not have a bottom element.

In the remainder of this section, we start by verifying that axioms
\ref{eq:8} and \ref{lifting} are also sound for representable $\{ -,
\rest\}$-algebras, thereby showing that every representable $\{-,
\rest\}$-algebra has a restriction semilattice structure.
We will also derive some algebraic consequences of \ref{schein1} --
\ref{lifting} that will be useful in the sequel.
\begin{lemma}\label{axioms}
  Axioms~\ref{eq:8} and \ref{lifting} are sound for all $\{ -,
  \rest\}$-algebras
  of partial functions and therefore for all representable $\{ -,
  \rest\}$-algebras.
\end{lemma}
\begin{proof}
  For~\ref{eq:8}, we observe that $(x,y) \in (a \rest c) \bmeet (b
  \rest c)$ exactly when $x$ belongs both to the domain of $a$ and to
  that of $b$, and $(x, y) \in c$. But having $x \in \dom(b)$ amounts
  to having $(x,z) \in b$ for some $z$, and thus we may conclude that
  \[\dom(a) \cap \dom(b) = \dom(a \rest b).\]
  This leads to the desired equality.

  Finally, for~\ref{lifting}, suppose $(x, y) \in (a \bmeet b) \rest
  a$. Then $(x, y) \in a$ and $x$ is in the domain of~$a \bmeet b$. By
  the later fact, there is some $z$ with $(x, z)$ in both $a$ and
  $b$. But since $a$ is a function $z$ must equal $y$. Hence $(x, y)
  \in a \bmeet b$. Conversely, if $(x, y) \in a \bmeet b$, then
  clearly $x \in \dom(a \bmeet b)$ and $(x, y) \in a$, so that $(x, y)
  \in (a \bmeet b) \rest a$.  
\end{proof}
We observe that axioms \ref{schein1} -- \ref{eq:8} are valid not
only for functions, but for arbitrary binary relations. However, the
validity of \ref{lifting} relies on $a$ being a function.

\begin{proposition}
  In a restriction semilattice, the following hold:
  \begin{align}
    \label{restricts}b \rest a &\leq a\\
    \label{eq:12} a \rest (b \rest c) & = (a \rest b) \rest c \\
    \label{eq:4} (a \rest b) \rest (a \bmeet b) & = a \bmeet b\\
    \label{eq:14} a \rest (b \bmeet c)& = (a \rest b) \bmeet c \\
    \label{left}(a \leq b, \ c \leq d)\limplies a \rest c &\leq b \rest d
  \end{align}
  In a difference--restriction algebra, we also have:
  \begin{equation}
    \label{eq:13} (a \rest b) - c  = a \rest (b - c)
  \end{equation}
\end{proposition}
\begin{proof}
  First we observe that the following equality holds:
  \begin{equation}
    \label{eq:18}
    (a \rest  (b \rest c))\bmeet (b \rest  (b \rest c))\bmeet (c \rest
    (b \rest c)) = (a \rest c) \bmeet (b \rest c).
  \end{equation}
  Indeed, by successively using~\ref{eq:8}, we can rewrite the
  left-hand side as
  \[((a \rest b) \rest (b \rest c))\bmeet (c \rest (b \rest c)) = ((a
    \rest b)\rest c) \rest (b \rest c) = ((a \rest c)\bmeet (b \rest
    c)) \rest (b \rest c),\]
  and by commutativity of $\cdot$ and~\ref{lifting}, this is precisely $(a
  \rest c) \bmeet (b \rest c)$.

  In what follows, we will freely use that $\bmeet$ is a semilattice
  operation, that is, $\bmeet$ is idempotent and commutative. We will
  also use that $\rest$ is idempotent, which is a consequence of
  \ref{lifting} and idempotency of the $\bmeet$ operation.
  \begin{description}
  \item[\eqref{restricts}:] This inequality translates into the
    equality $a \bmeet (b \rest a) = b \rest a$. Taking $a = c$
    in~\eqref{eq:18}, we have:
    \begin{align*}
      a \bmeet (b \rest a)
      & = (a\rest a) \bmeet (b \rest a) = (a \rest  (b \rest a))\bmeet
        (b \rest  (b \rest a))\bmeet (a \rest(b \rest a))
      \\ & = (b \rest  (b \rest a))\bmeet (a \rest(b \rest a)) \just
           {\ref{eq:8}} = (b \rest a) \rest (b \rest a) = (b \rest
           a).
    \end{align*}
  \item[\eqref{eq:12}:] By~\ref{eq:8}, $(a \rest b) \rest c$
    equals the left-hand side of~\eqref{eq:18}. We prove that so does
    $a \rest (b \rest c)$:
    \[(a \rest (b \rest c)) \bmeet (b \rest (b \rest c)) \bmeet (c
      \rest (b \rest c)) \just {\ref{eq:8}} = (a \rest (b \rest c))
      \bmeet ((b \rest c) \rest (b \rest c)) \just{\eqref{restricts}}=
      a \rest (b \rest c).\]
  \item[\eqref{eq:4}:] We first observe that $b \rest (a \rest b) = a
    \rest b$. Indeed, we may compute:
    \begin{align}\label{eq:2.11}
      b \rest (a \rest b)
      & \just{\eqref{eq:12}} = (b \rest a)\rest b \just{\ref{eq:8}}
        = (b\rest b)\bmeet (a \rest b) \just {\eqref{restricts}}=
        a\rest b.
    \end{align}
    Then we can compute:
    \begin{align*}
      (a \rest b) \rest
      (a \bmeet b) &\just{\ref{eq:8}}=(a \rest (a \bmeet b))\bmeet(b \rest (a \bmeet b)) \\&\just{\ref{lifting}}= (a \rest ((a \bmeet b) \rest a)) \bmeet (b \rest ((a \bmeet b) \rest b)) \\&\just{\eqref{eq:2.11}}= ((a\bmeet b) \rest a) \bmeet ((a\bmeet b) \rest
      b)
      \\ &\just{\ref{lifting}}=a \bmeet b.
    \end{align*}
  \item[\eqref{eq:14}:] This follows from:
    \begin{align*}
      a \rest (b \bmeet c)
      & \just{\ref{lifting}} = a \rest ((b \bmeet c) \rest b)
        \just {\eqref{eq:12}} = (a \rest (b\bmeet c)) \rest b
      \\ & \just {\ref{eq:8}} = (a \rest b) \bmeet ((b \bmeet c) \rest b)
           \just {\ref{lifting}} = (a \rest b) \bmeet (b \bmeet c) \just
           {\eqref{restricts}} = (a \rest b) \bmeet c.
    \end{align*}
  \item[\eqref{left}:] Let $a\leq b$ and $c \leq d$. The desired
    inequality is a consequence of combining the inequalities $a\rest
    c \leq b \rest c$ and $b \rest c \leq b \rest d$. While the former
    is a trivial consequence of~\ref{eq:8} and the fact that $a\rest
    b = a$ implied by~\ref{lifting}, the latter may be derived as
    follows:
    \begin{align*}
      (b \rest c)\bmeet (b \rest d)
      & \just{\ref{lifting}} = (b \rest (c\rest d))\bmeet (b \rest
        d) \just{\eqref{eq:12}} = ((b \rest c)\rest d)\bmeet (b \rest
        d)
      \\ & \just{\ref{eq:8}} = (b \rest(b \rest
           c))\rest d \just{\ref{lifting}, \eqref{eq:12}} = b
           \rest (c\rest d) \just{\ref{lifting}} = b \rest c.
    \end{align*}
  \item[\eqref{eq:13}:] We have:
    \begin{align*}
      (a \rest b) - c &\just{\hspace{2pt}\eqref{restricts}\hspace{2pt}} = (b \bmeet (a \rest b)) - c
                        \just{\eqref{complement}} = (b - (b - (a \rest b))) - c\\
                      &\just{\ref{schein3}} = (b - c) - (b - (a \rest b))
                        = (b - c) - ((b - c) - (a \rest b))\\
                      &\just{\hspace{5.7pt}\eqref{complement}\hspace{5.7pt}} = (a \rest b) \bmeet (b - c)\just {\eqref{eq:14}} = a \rest (b \bmeet (b - c)) = a \rest (b - c),
    \end{align*}
    where the unmarked equalities follow from evident properties of
    sets with relative complement.\qedhere
  \end{description}
\end{proof}

Since, by property~\eqref{eq:12}, the operation $\rest$ is
associative, from here on we may write $a \rest b \rest c$ instead of
$(a \rest b) \rest c$ or $a \rest (b \rest c)$.

For any poset $\cS$ and $a \in \cS$, the notation $a^\downarrow$
denotes the set $\{b \in \cS \mid b \leq a\}$. It is known that in any subtraction algebra $\cS$, for any $a \in \cS$, the set $a^\downarrow$, with least element
  $0$, greatest element $a$, meet given by $\bmeet$ and
  complementation given by $\compl{b} \coloneqq a - b$ is a Boolean
  algebra \cite[page~2154]{schein1992difference}. We note the following corollary.

\begin{corollary}\label{c:3}
  If $h: \cS \to \cT$ is a homomorphism of subtraction algebras then, for every $a \in \cS$,
  the map $h$ induces a homomorphism of Boolean algebras $h_a:
  a^\downarrow \to h(a)^\downarrow$.
\end{corollary}
  
We can also prove a sort of converse to each $a^\downarrow$ of a
subtraction algebra being a Boolean algebra.

\begin{proposition}\label{p:1}
  Suppose that $(\cS, \cdot)$ is a meet-semilattice with bottom~$0$
  such that, for every $a \in \cS$, there is a unary operation
  $\overline{\phantom a}^a$ on $a^{\downarrow}$ such that $(a^{\downarrow},  0,
  a, \cdot, \overline{\phantom a}^a)$ is a Boolean algebra. Then setting $a-b :=
  \overline{a \cdot b}^a$ defines a subtraction algebra structure on
  $\cS$ (on which \eqref{complement} becomes a valid equation).
\end{proposition}
\begin{proof}
First we argue that  $\bmeet$ is the operation obtained from the term $a - (a - b)$. Now $(a -b)$ is by definition in $a^\downarrow$, so $a - (a -b) = \overline{a \cdot (a-b)}^a = \overline{a-b}^a$, which is the complement of the complement of $a \bmeet b$ in $a^\downarrow$, that is, equals $a \bmeet b$. As a consequence, the validity of \ref{commutative}, which is formally a statement about $-$, is immediate.

For the validity of \ref{schein1}, the term $a - (b - a)$ is by definition $\overline{a \bmeet\overline{b\bmeet a}^b}^a$. We calculate $a \bmeet\overline{b\bmeet a}^b$. Now $\overline{b\bmeet a}^b$ is by definition less than or equal to $b$, so $a \bmeet\overline{b\bmeet a}^b = a \bmeet b\bmeet\overline{b\bmeet a}^b$. But $\overline{b\bmeet a}^b$ is the complement of $a \bmeet b$ in $b^\downarrow$, so $a \bmeet b\bmeet \overline{b\bmeet a}^b = 0$. Hence $a - (b - a) = \overline{0}^a = a$.

For the validity of \ref{schein3}, we start from the term $(a-b) - c$. We may assume $b$ and $c$ are in $a^\downarrow$ since $a-b = \overline{a \bmeet b}^a = \overline{a \bmeet a \bmeet b}^a = a - (a \bmeet b)$ and $(a-b) - c = \overline{(a-b) \bmeet c}^a = \overline{(a-b) \bmeet a \bmeet c}^a = (a-b) - (a \bmeet c)$. We write simply $\overline{\phantom c}$ for the complement in $a^\downarrow$, and $+$ for the join. Then $(a - b) - c = \overline b - c = \overline{\overline b \bmeet c}^{\overline b}$. In any Boolean algebra, the complement operation on the induced Boolean algebra $d^\downarrow$ is given by $e \mapsto  d \bmeet \overline e$. So as $\overline b$ is an element of the Boolean algebra $a^\downarrow$, the complement on $\overline{b}^\downarrow$ is given by $d \mapsto \overline b \bmeet \overline d$. Hence $\overline{\overline b \bmeet c}^{\overline b} = \overline b \bmeet (b + \overline c) = \overline b \bmeet \overline c$. It is now clear, by commutativity of $\bmeet$ and symmetry, that this is equal to $(a-c) - b$. 
\end{proof}

\section{The algebra of domains}\label{sec:domains}

Throughout this section, let $\algebra A$ be a
restriction semilattice. Although we do not have the \emph{domain
  operation} in our signature, the signature is expressive enough that
it can express the `domain inclusion' relation. In this section, we
will begin to investigate this implicit domain information.

\begin{definition}\label{sec:1}
  Define the relation $\preceq_\cA$ on $\algebra A$ by $a \preceq_\cA
b \iff a \leq b \rest a$.
\end{definition} We often drop the subscript $\algebra A$. Notice
that, by~\eqref{restricts}, we have $a \preceq b$ if and only if $a =
b \rest a$.  When $\cA$ is an algebra of partial functions it is not
hard to see that, for every $f, g \in \cA$, we have $f \preceq g$
exactly when $\dom(f) \subseteq \dom(g)$.

\begin{lemma}\label{l:3}
  The following statements hold:
  \begin{enumerate}[label = (\alph*)]
  \item\label{item:5} $\preceq$ is a preorder on $\cA$ that contains
    $\leq$;
  \item\label{item:7} if $\cA$ has a bottom element~$0$, then for
    every $a \in \cA$, if $a \preceq 0$, then $a = 0$.
  \end{enumerate}
\end{lemma}
\begin{proof}
  The fact that $\preceq$ is reflexive, that is, that $a = a \rest a$
  for every~$a \in \cA$, follows from~\ref{lifting}. To prove
  that $\preceq$ is transitive, let $a,b, c \in \cA$ be such that $a
  \preceq b$ and $b \preceq c$. Then we have
  \begin{align*}
    c \rest a \just{a \preceq b} = c \rest b \rest a \just{b \preceq
    c} = b \rest a \just {a \preceq b} = a,
  \end{align*}
  and thus $a \preceq c$.
  To see that $\preceq$ contains $\leq$ let $a, b \in \cA$ be such
  that $a \leq b$. Since, by~\eqref{left}, the operation $\rest$ is
  order preserving, we have $a = a \rest a \leq b \rest a$, and hence $a
  \preceq b$. This proves~\ref{item:5}. For~\ref{item:7}, we observe
  that $a \preceq 0$ means $a \leq 0\rest a$. But,
  by~\ref{lifting}, we have $0 \rest a = 0$, and thus $a = 0$.
\end{proof}

Observe that every homomorphism of $\{\cdot, \rest\}$-algebras is
$\preceq$-preserving, since $\preceq$ is defined by an equation.

We denote by $\sim_\cA$ the equivalence relation induced by
$\preceq_\cA$, and for a given $a \in \cA$ we use $[a]$ to denote the
equivalence class of~$a$. The canonical projection $\cA
\twoheadrightarrow \cA/{\sim_\cA}$ is denoted by~$\pi_\cA$. As with $\preceq$, if $\cA$
is clear from context, we denote $\sim_\cA$ and
$\pi_\cA$ by $\sim$ and $\pi$, respectively.  Given a
subset $S \subseteq \cA$, we may use $S/{\sim}$ to denote the forward
image $\pi[S]$ of~$S$ under~$\pi$.

\begin{lemma}\label{l:2}
  The poset $\algebra A / {\sim}$ of $\sim$-equivalence classes (with order inherited from $\preceq$) is a
  meet-semilattice with
  \begin{equation}
    [a] \wedge [b] = [a \rest b].\label{eq:2}
  \end{equation}
  In particular, we have $a^\downarrow/{\sim} = [a]^{\downarrow}$.
\end{lemma}
\begin{proof}
  The fact that $[a \rest b]$ is a lower bound of $[a]$, that is, that
  the inequality $a \rest b \leq a \rest a \rest b$ holds, is a
  consequence of having $a \rest a = a$ by~\ref{lifting}. In turn, $[a
  \rest b]$ is a lower bound of $[b]$ thanks to~\eqref{restricts} and
  the fact that $\leq$ is contained in $\preceq$, by \Cref{l:3}.
  Suppose that $[c] \preceq [a], [b]$, that is, $c = a \rest c = b
  \rest c$. Then we may compute
  \[a \rest b \rest c = a \rest c = c,\]
  hence $[c] \preceq [a \rest b]$ and we have~\eqref{eq:2}.

  For the second assertion, by \Cref{l:3} it is clear that
  $a^\downarrow/{\sim} \subseteq [a]^{\downarrow}$. Conversely, let
  $[b] \preceq [a]$. By~\eqref{eq:2}, we have $[b \rest a] = [a]
  \wedge [b] = [b]$, and by~\eqref{restricts} we have $b \rest a \leq
  a$. Thus $[b] \in a^\downarrow/{\sim}$.
\end{proof}

\begin{lemma}\label{l:5}
  The relations $\leq$ and $\preceq$ coincide on each downset
  $a^{\downarrow}$, for $a \in \cA$. In particular, $[a]^\downarrow$
  is order isomorphic to $a^{\downarrow}$.  Explicitly, for every $[b]
  \preceq [a]$, the element $b \rest a$ is the unique element of
  $a^\downarrow$ that is $\sim$-equivalent to~$b$.
\end{lemma}
\begin{proof}
  By \Cref{l:3}, we already know that $\preceq$ contains
  $\leq$. Conversely, let $x, y \in a^\downarrow$ be such that $x
  \preceq y$, that is, $y \rest x = x$. Using~\ref{lifting},
  we have $x = x \rest a$ and $y = y \rest a$. Therefore, $x\bmeet y =
  (y \rest a)\bmeet (x \rest a)$, and by~\ref{eq:8} it follows that
  $x \bmeet y = (y \rest x) \rest a$. Using the hypothesis that $x
  \preceq y$ and $x \rest a = x$, we may conclude that $x \leq y$ as
  intended. 
  
   Thus $a^\downarrow$ is order isomorphic to $a^\downarrow / {\sim}$, which by the second part of \Cref{l:2} equals $[a]^\downarrow$. For the last assertion, we saw in the proof of \Cref{l:2} that $b \rest a$ is in $a^\downarrow$ and $\sim$-equivalent to $b$, thus we now know it is the unique such element of $a^\downarrow$.
\end{proof}
\begin{corollary} \label{c:2} 
If $a \leq b$ and $b \preceq a$, then $a = b$. Hence, $a
  < b$ implies $a \prec b$, where $\prec$ denotes the strict relation
  derived from $\preceq$ (that is, $a \prec b$ if and only if $a \preceq
  b$ and $a \not\sim b$).
\end{corollary}
\begin{proof}
  Pick two elements $a, b$ validating $a \leq b$ and $b \preceq a$. By
  \Cref{l:3}, we know $a$ and $b$ are $\sim$-equivalent, and moreover they are
  both in the downset~$b^\downarrow$. Thus, by \Cref{l:5}, they must be equal.
\end{proof}

For the last result of this section we will assume that $\cA$ is a
difference--restriction algebra.

\begin{corollary}\label{c:1}
  The poset $\cA/{\sim}$ admits a subtraction algebra structure, where the
  operation $-$ is given by
  \[[a]-[b] = [a- (b \rest a)].\]
\end{corollary}
\begin{proof}
  Recall that, because $\cA$ admits a subtraction algebra structure,
  for every $a \in \cA$, the poset $a^\downarrow$ is a Boolean algebra
  (see paragraph before \Cref{c:3}).  By \Cref{l:5} we know that, for
  every $a \in \cA$, the poset $[a]^\downarrow$ is a Boolean algebra
  isomorphic to $a^\downarrow$ via the assignment $[b] \mapsto b \rest
  a$. So by \Cref{p:1}, $\cA/{\sim}$ is a subtraction algebra with
  $[a] - [b]$ equal to the complement of $[a]\wedge[b] = [b \rest a]$
  in $[a]^\downarrow$ (recall Lemma~\ref{l:2}). This is indeed $[a-(b
  \rest a)]$, as $b \rest a \in a^\downarrow$ by~\eqref{restricts}.
\end{proof}

\section{Filters}\label{sec:filters}

We continue to let $\cA$ denote a restriction semilattice. Since $\cA$
is in particular a semilattice, the notion of a \emph{filter of
  $\cA$} is well defined.

\begin{definition}
  A subset $F$ of a meet-semilattice $\cS$ is a \defn{filter} if
  \begin{enumerate}[(i)]
  \item $F$ is nonempty,
  \item $F$ is upward closed,
  \item $a, b \in F \implies a \wedge b \in F$.
  \end{enumerate}
  We use $\filt \cS$ to denote the set of all filters of $\cS$.
\end{definition}

We now concentrate on $\filt \cA$. First observe that there is a
natural embedding of sets
\[\iota: \cA \hookrightarrow \filt\cA,\qquad a \mapsto a^\uparrow.\]
For this reason, we will often treat $\cA$ as a subset of $\filt\cA$. 
The operations $\cdot$ and $\rest$ on~$\cA$ may naturally be extended
to operations $\cdot_{\mathsf F}$ and $\rest_{\mathsf F}$ on $\filt{\cA}$ as follows.
Given $F, G \in \filt\cA$ we set
\[F \cdot_{\mathsf F} G := \langle F \cdot G\rangle_{\sf Filt}
  \quad\text{and}\quad F \rest_{\mathsf F} G := \langle F \rest G\rangle_{\sf
    Filt},\]
where $\langle \phantom a\rangle_{\sf
    Filt}$ denotes the well-defined operation `filter generated by\dots'.
We observe now that $\cdot_{\mathsf F}$ and $\rest_{\mathsf F}$ are indeed extensions of $\cdot$ and
$\rest$, respectively. Indeed, while it is easily seen that
$a^\uparrow \cdot_{\mathsf F} b^{\uparrow} = (a \cdot b)^{\uparrow}$, the
equality $a^\uparrow \rest_{\mathsf F} b^{\uparrow} = (a \rest b)^{\uparrow}$
follows from~\eqref{left}. Explicitly, $\cdot_{\mathsf F}$ and $\rest_{\mathsf F}$ are
given as follows.
\begin{lemma}\label{l:1}
  For every $F, G \in \filt\cA$, the following equalities hold:
  \[F \cdot_{\mathsf F} G = (F \cdot G)^{\uparrow}\quad \text{and}\quad F
    \rest_{\mathsf F} G = (F \rest G)^{\uparrow}.\]
\end{lemma}
\begin{proof}
  It is clear that $F \cdot_{\mathsf F} G$ and $F \rest_{\mathsf F} G$ contain $(F
  \cdot G)^{\uparrow}$
  and $(F \rest G)^{\uparrow}$, respectively. Thus it suffices to show that $(F
  \cdot G)^{\uparrow}$ and $(F \rest G)^{\uparrow}$ are filters. Since
  $F$ and $G$ are filters with respect to the semilattice operation
  induced by $\cdot$, it is a standard result that $(F \cdot G)^{\uparrow}$ is
  precisely the filter generated by $F \cup G$. Let us see that $(F
  \rest G)^\uparrow$ is also a filter.
\begin{enumerate}[(i)]
\item As $F$ and $G$ are nonempty, $F \rest G$ is nonempty, and therefore $(F
  \rest G)^\uparrow$ is nonempty too.
\item The set $(F \rest G)^\uparrow$ is upward closed by definition.
\item Suppose $x, y \in (F \rest G)^\uparrow$. So there are $a, b \in
  F$ and $c,d \in G$ with $x \geq a \rest c$ and $y \geq b \rest
  d$. As $F$ and $G$ are filters, we know $a \bmeet b \in F$ and $c
  \bmeet d \in G$. By~\eqref{left}, we find that $(a \bmeet b) \rest
  (c \bmeet d) \leq (a \rest c) \bmeet (b \rest d)$. It follows that
  the element $(a \bmeet b) \rest (c \bmeet d)$ of $F \rest G$ is
  smaller than or equal to $x \bmeet y$. Hence $x \bmeet y \in (F
  \rest G)^\uparrow$. \qedhere
\end{enumerate}
\end{proof}

As observed in the proof of Lemma~\ref{l:1}, the set $F \cdot_{\mathsf F}
G$ is the filter generated by $F \cup G$. Therefore $\cdot_{\mathsf F}$ is a
semilattice operation on~$\filt\cA$ whose induced order~$\lleq$ is
reverse inclusion, that is,
\begin{equation}
  \label{eq:3}
  F \lleq G \iff F \supseteq G.
\end{equation}
In particular, $\filt\cA$ has bottom element~$\cA$, the full filter.

We will now see that $(\filt \cA, \cdot_{\mathsf F}, \rest_{\mathsf F})$ is again a
restriction semilattice, and thus all the results of
Section~\ref{sec:domains} (except Corollary~\ref{c:1}) also hold for
the filter algebra of a restriction semilattice.

\begin{proposition}\label{p:2}
  The operations $\cdot_{\mathsf F}$ and $\rest_{\mathsf F}$ endow $\filt\cA$ with a
  restriction semilattice structure.
\end{proposition}
\begin{proof}
  We already observed that $(\filt\cA, \cdot_{\mathsf F})$ is a semilattice. It
  remains to prove that~\ref{eq:8} and~\ref{lifting} hold. We note
  that, if $\bullet$ is a binary operation on~$\cA$ that is order
  preserving on both coordinates, then for every $S_1, S_2 \subseteq
  \cA$, we have $(S_1^{\uparrow} \mathbin\bullet S_2^{\uparrow})^{\uparrow} =
  (S_1 \mathbin\bullet S_2)^{\uparrow}$. Both $\cdot$ and $\rest$ \emph{are} order
  preserving on both coordinates: for $\cdot$ this is a simple consequence of the definition of $\leq$ in terms of~$\cdot$, and for $\rest$ this
  follows from~\eqref{left}. Thus the noted equality holds for both $\cdot$ and $\rest$, and this observation will be freely used in
  the rest of the proof. 
  
  Let $F, G, H \in \filt\cA$.
  \begin{description}
  \item[\ref{eq:8}:] We need to show that $(F \rest_{\mathsf F} H) \cdot_{\mathsf F}(G
    \rest_{\mathsf F} H) = (F \rest_{\mathsf F} G) \rest_{\mathsf F} H$. The inclusion $\supseteq$
    follows easily from Lemma~\ref{l:1} and~\ref{eq:8} for the
    algebra~$\cA$. Conversely, let $a \in F$, $b\in G$ and $c, c' \in
    H$.  Then,
    \[(a \rest c) \cdot (b \rest c') \just {\eqref{left}}\geq (a \rest
      (c\cdot c')) \cdot (b \rest (c\cdot c')) \just{\ref{eq:8}} = (a
      \rest b) \rest (c\cdot c').\]
    Since $H$ is a filter, we have $c\cdot c' \in H$ and thus, $(a
    \rest c) \cdot (b \rest c') \in ((F \rest G) \rest H)^{\uparrow} =
    (F \rest_{\mathsf F} G) \rest_{\mathsf F} H$.
  \item[\ref{lifting}:] The goal is to show that $(F \bmeet_{\mathsf F} G)
    \rest_{\mathsf F} F = F \bmeet_{\mathsf F} G$. Again, the inclusion~$\supseteq$ is a
    straightforward consequence of Lemma~\ref{l:1} and~\ref{lifting} for~$\cA$. Conversely, given $a,a' \in F$ and
    $b \in G$, we have
    \[(a \cdot b)\rest a' \just{\eqref{left}}\geq ((a \cdot a') \cdot
      b) \rest (a \cdot a') \just{\ref{lifting}} = (a \cdot a') \cdot
      b.\]
    Since $F$ is a filter, we may then conclude that $(a \cdot b)\rest
    a'$ belongs to $(F \cdot G)^{\uparrow} = F \cdot_{\mathsf F} G$, as
    required. \qedhere
  \end{description}
\end{proof}

We will denote by $\ppreceq$ the relation $\preceq_{\filt\cA}$
obtained by applying \Cref{sec:1} to the restriction semilattice $(\filt \cA, \cdot_{\mathsf F}, \rest_{\mathsf F})$. Using Lemma~\ref{l:1},
equivalence~\eqref{eq:3}, and the fact that every filter is upward
closed, we have
\begin{equation}
  F \ppreceq G \iff F \supseteq G \rest F,\label{eq:5}
\end{equation}
for every $F, G \in \filt\cA$. Note that, since $\ppreceq =
\,\preceq_{\filt\cA}$ is defined by a $\{\cdot_{\mathsf F}, \rest_{\mathsf F}\}$-equation,
the relation $\ppreceq$ on $\filt \cA$ is an extension of the relation
$\preceq$ on~$\cA$.
By Lemma~\ref{l:3}\ref{item:5}, the relation $\ppreceq$ is a preorder on
$\filt\cA$ that contains $\supseteq$. We denote by $\approx$ the
equivalence relation induced by $\ppreceq$, and by
\[\rho: \filt\cA \to \filt\cA/{\approx}\]
the canonical projection. The $\approx$-equivalence class of a filter
$F \in \filt \cA$ is denoted~$\dbracket F$. By Lemma~\ref{l:2},
the poset $\filt\cA /{\approx}$, with order inherited from $\ppreceq$, also admits a meet-semilattice structure, with
the meet of two filters $F, G$ given by
\begin{equation}
  \label{eq:1}
  \dbracket F \wedge \dbracket G = \dbracket{F \rest_{\mathsf F}
    G}.
\end{equation}
We finish this section by showing that $\ppreceq$ admits a description
in terms of the projection~$\pi: \cA \twoheadrightarrow \cA/{\sim}$.
\begin{proposition}\label{p:3}
  For every filter $F \subseteq \cA$, the subset $\pi[F]^{\uparrow}
  \subseteq \cA/{\sim}$ is a filter of $\cA/{\sim}$, and conversely,
  every filter of $\cA/{\sim}$ is of the form $\pi[F]^{\uparrow}$ for
  some filter $F \subseteq \cA$. Moreover, for every $F, G \in
  \filt\cA$, we have
  \begin{equation}
    \label{eq:15}
    F \ppreceq G \iff \pi[G]^\uparrow \subseteq \pi[F]^\uparrow.
  \end{equation}
  In particular, the surjections
  \[\rho: \filt\cA \twoheadrightarrow \filt\cA/{\approx}\qquad \text{
      and }\qquad \pi[\_]^\uparrow: \filt\cA \twoheadrightarrow
    \filt{\cA/{\sim}}\] are isomorphic (in the sense that there is a bijection $f$ such that $f \circ \rho = \pi[\_]^\uparrow$).
\end{proposition}
\begin{proof}
  First we show that $\pi[F]^\uparrow = (F / {\sim})^\uparrow$ is a
  filter.
  \begin{enumerate}[(i)]
  \item As $F$ is nonempty, $F / {\sim}$, and therefore $(F /
    {\sim})^\uparrow$, is nonempty.
\item The set $(F / {\sim})^\uparrow$ is upward closed by definition.

\item Suppose $[a], [b] \in (F /{\sim})^\uparrow$, say $[a_0] \preceq
  [a]$ and $[b_0] \preceq [b]$ for $a_0, b_0 \in F$. Then as~$F$ is a
  filter, we have $a_0 \bmeet b_0 \in F$. Using~\eqref{eq:4} and
  \Cref{l:2}, we have
  \[[a_0 \bmeet b_0] \preceq [a_0 \rest b_0] = [a_0] \wedge [b_0],\]
  and thus $[a_0] \wedge [b_0]$ belongs to $(F/{\sim})^{\uparrow}$,
  yielding that so does $[a] \wedge [b]$, as desired.
\end{enumerate}

Conversely, let $G \subseteq \cA/{\sim}$ be a filter. Define the
relation $\equiv$ on $\pi^{-1}(G)$ by $a \equiv b \iff \exists d \in
\pi^{-1}(G) : d \rest a = d \rest b$. Then $\equiv$ is clearly
reflexive and symmetric. It is also transitive, because given $d_1,
d_2 \in \pi^{-1}(G)$ such that $d_1 \rest a = d_1 \rest b$ and $d_2
\rest b = d_2 \rest c$, we have $d_1 \rest d_2 \in \pi^{-1}(G)$ and
$d_1 \rest d_2 \rest a = d_1 \rest d_2 \rest c$ (proven using the law
$d_1 \rest d_2 \rest x = d_2 \rest d_1 \rest x$, which is an evident
consequence of \ref{eq:8}). Take any equivalence class $E$ of
$(\pi^{-1}(G), \equiv)$---as $\pi^{-1}(G)$ is nonempty, there exists
at least one choice. We claim that $E$ is a filter. It is nonempty by
definition. Let us see that it is upward closed. Given $a \in E$ and
$b \in \cA$ such that $a \leq b$, since $\pi^{-1}(G)$ is upward closed
(as $\pi$ is order preserving), we have that $b \in \pi^{-1}(G)$. To
conclude that $b$ belongs to $E$, we only need to show that $a$ and
$b$ are $\equiv$-equivalent. That is indeed the case because,
successively using~\ref{lifting} and the inequality $a \leq b$ (as
well as idempotency and commutativity of $\cdot$), we may compute
\[a \rest a = a = a \cdot b = (a\cdot b) \rest b = a \rest b.\]
Finally, if $a, b \in E$, then taking $d \in \pi^{-1}(G)$ such that $d
\rest a = d \rest b$, we know $\pi(d \rest a) \in G$, so $d \rest a
\in \pi^{-1}(G)$. Then as $d \rest (d \rest a) = d \rest a$ and $a \in
E$, we have $d \rest a \in E$. Then since $d \rest a = d \rest b \leq
a, b$, the upward-closed set $E$ is closed under meets. Hence $E$ is a
filter. Then, given $[d] \in G$, if we choose some $a \in E$ we have
$d \rest a \in E$ and $\pi(d \rest a) \preceq [d]$, so $[d] \in
\pi[E]^\uparrow$. Hence $\pi[E]^\uparrow = G$.

Finally, we prove~\eqref{eq:15}. First assume that $F \ppreceq G$. Take an arbitrary $a \in G$. As $F$ is nonempty, we can choose some $b \in F$. As $F \ppreceq G$, we have $a \rest b \in G \rest F \subseteq F$. So $[a \rest b] \in \pi[F]$. Since $[a \rest b] \preceq [a]$, we obtain $[a] \in \pi[F]^\uparrow$. As $a \in G$ was arbitrary, we deduce $\pi[G] \subseteq \pi[F]^\uparrow$. Thus $\pi[G]^\uparrow \subseteq \pi[F]^\uparrow$.

Conversely, suppose we have $\pi[G]^\uparrow \subseteq
\pi[F]^\uparrow$, and pick any $a \in G$ and $b\in F$, so that $a
\rest b \in G \rest F$. Since $[a]$ and $[b]$ both belong to
$\pi[F]^\uparrow$, so does $[a] \wedge [b] = [a \rest b]$. Hence there
exists $c \in F$ such that $c \preceq a \rest b$.  In turn, by
reflexivity of ${\ppreceq}$, the element $c \rest b$ also belongs to
$F$, and furthermore we have $c \rest b \preceq c\preceq a \rest
b$. Thus $c \rest b \preceq a \rest b$ in $b^\downarrow$, and by
\Cref{l:5} this implies $c \rest b \leq a \rest b$, which yields $a
\rest b \in F$ as required.
\end{proof}

\section{Representability}\label{sec:rep}

In this section, we finally show that equations \ref{schein1} --
\ref{lifting} axiomatise the representable $\{-, \rest\}$-algebras, by
describing a generic representation, by partial functions, of any such
algebra.

Recall that a filter $F \subseteq \cS$ of a meet semilattice $\cS$ is
\defn{maximal} if it is proper (not the whole of $\cS$) and is maximal
amongst all proper filters of $\cS$, with respect to inclusion. We use
$\mf \cS$ to denote the set of all maximal filters of~$\cS$.

\begin{lemma}\label{filter_up}
  Let $\cA$
  be a restriction semilattice. Then for every maximal filter $\mu
  \subseteq \cA$ and for every filter $F \subseteq \cA$, the
  following are equivalent:
  \begin{enumerate}[label = (\alph*)]
  \item $\mu \rest_{\mathsf F} F$ is maximal,
  \item\label{item:3} $\mu \ppreceq F$,
  \item\label{item:4} $\mu \approx \mu \rest_{\mathsf F} F$.
  \end{enumerate}
\end{lemma}
\begin{proof}
  To show that $\mu \rest_{\mathsf F} F$ is maximal precisely when $\mu \ppreceq
  F$, we first assume $\mu \rest_{\mathsf F} F$ (that is, $(\mu \rest F)^\uparrow$) is
  maximal. We want to show that $(F \rest \mu)^\uparrow = \mu$. Because,
  by~\eqref{restricts}, $\mu \subseteq (F \rest \mu)^\uparrow$, and
  because $\mu$ is maximal, it suffices to show that the filter $(F \rest
  \mu)^\uparrow$ is proper. Suppose not: then there is $a \in F$ and
  $b \in \mu$ such that $a \rest b = 0$. By \Cref{l:2}, we obtain $0
  \sim a \rest b \sim b \rest a$. So by \Cref{l:3}\ref{item:7}, we
  have $0 = b \rest a \in \mu \rest F$, contradicting properness of
  $(\mu \rest F)^\uparrow$.
  Conversely, we have $\mu \ppreceq F$ if and only if $F \rest \mu
  \subseteq \mu$, which implies $\mu \rest F \rest \mu \subseteq \mu
  \rest \mu$. Since $\mu$ is closed under $\cdot$, by~\ref{lifting}
  and~\eqref{left}, we have $\mu \rest \mu \subseteq \mu$. So if $0$
  belonged to $\mu \rest F$, it would also belong to $\mu$,
  contradicting the properness of $\mu$. (The law $0 \rest a = 0$
  follows also from \ref{lifting}.) Thus $(\mu \rest F)^\uparrow$ is
  proper. 
  Now, to show $(\mu \rest F)^\uparrow$ is maximal, suppose $G$ is a
  filter with $(\mu \rest F)^\uparrow \subseteq G \subsetneq \cA$. If
  $G \rest_{\mathsf F} \mu = \cA$ then there are $a \in G$ and $b \in
  \mu$ with $a \rest b = 0$. In that case, making use of an element $c
  \in F$, we have
  \[0 = a \rest b \rest c \geq ((b \rest c) \cdot a) \rest (b \rest c)
    = (b \rest c) \cdot a \in G,\]
  where the inequality follows from \eqref{left} and the last equality
  from \ref{lifting}. This contradicts $G \subsetneq \cA$; hence $G
  \rest_{\mathsf F} \mu \neq \cA$. Since $G \rest_{\mathsf F} \mu
  \supseteq \mu$ and $\mu$ is maximal, we therefore have $G
  \rest_{\mathsf F} \mu = \mu$. Now given an arbitrary $a \in G$, take
  $b \in \mu$ and $c \in F$. Then since $F \subseteq \mu
  \rest_{\mathsf F} F \subseteq G$, we have $a \cdot c \in G$, so $(a
  \cdot c) \rest b \in G \rest_{\mathsf F} \mu = \mu$, and thus $(a
  \cdot c) \rest b \rest c \in \mu \rest F$. But $(a \cdot c) \rest b
  \rest c = b \rest (a \cdot c) \rest c = b \rest (a \cdot c) \leq a$,
  and hence $a \in (\mu \rest F)^\uparrow$. Hence $G = (\mu \rest
  F)^\uparrow$.

  Finally, \ref{item:3} and~\ref{item:4} are equivalent because,
  by~\eqref{eq:1}, we have \[\mu \ppreceq F \iff \dbracket \mu =  \dbracket \mu \wedge \dbracket F  \iff \dbracket \mu =  \dbracket{\mu \rest_{\mathsf F}
    F} \iff \mu 
  \approx \mu \rest_{\mathsf F} F.\qedhere\]
\end{proof}

\begin{proposition}\label{p:4}
  Let $\cA$ be a restriction semilattice, and let $\theta \from \cA \to \mathcal{PF}(\mf \cA)$ be the map
  given by
  \begin{equation*}
    a^\theta \coloneqq \{(\xi, \mu)\in \mf \cA \times \mf \cA \mid \xi
    \approx \mu \ \text{ and } \ a \in \mu \}.
  \end{equation*}
  Then $\theta$ is a homomorphism of restriction semilattices.
\end{proposition}
\begin{proof}
  Since (maximal) filters containing $a$ belong to the downset $\{F
  \in \filt\cA \mid F \lleq a^{\uparrow}\} = \{F \in \filt\cA \mid F
  \supseteq a^{\uparrow}\}$ (recall~\eqref{eq:3}), the fact that each
  $a^\theta$ is a partial function on $\mf \cA$ is an immediate
  consequence of Lemma~\ref{l:5} applied to the
  restriction semilattice $\filt\cA$.
  
  For showing that $\theta$ represents both operations correctly, we
  pick two $\approx$-equivalent maximal filters $\xi, \mu \in \mf
  \cA$.
  
  For $\cdot$ we have the following:
  \begin{align*}
    (\xi, \mu) \in a^\theta \cap b^\theta
    & \iff a \in \mu \ \text{ and }\ b \in \mu
    \\  & \iff a\cdot b \in \mu \qquad\qquad\qquad \text{(because $\mu$ is
          a filter)}
    \\  & \iff (\xi, \mu) \in (a\cdot b)^\theta.
  \end{align*}

   For $\rest$ suppose $(\xi, \mu) \in (a \rest
  b)^\theta$. Then $a \rest b \in \mu$, so $b \in \mu$, by
  \eqref{restricts}. Hence $(\xi, \mu) \in b^\theta$. To show that
  $(\xi, \mu) \in a^\theta \rest b^\theta$, it remains to show that
  $\xi$ is in the domain of the partial function $a^\theta$. That is,
  we must find a maximal filter $\nu$ with $a \in \nu$ and $\nu
  \approx \xi$. We claim that $(\mu \rest \{a\})^\uparrow$ (equal to
  $\mu \rest_{\mathsf F}a^{\uparrow}$) is the required $\nu$. We noted that $b
  \in \mu$; hence $b \rest a \in \mu \rest \{a\}$, and hence $(\mu
  \rest \{a\})^\uparrow$ contains $a$, by \eqref{restricts}. By
  \Cref{filter_up}, we have that $(\mu \rest \{a\})^\uparrow$ is a
  maximal filter $\approx$-equivalent to $\mu$, hence to $\xi$,
  provided $\mu \ppreceq a^\uparrow$. That is, provided $a^\uparrow \rest \mu
  \subseteq \mu$ (recall~\eqref{eq:5}). Since, $\mu$ is upward closed
  and, by~\eqref{left}, $\rest$ is order-preserving in the first
  coordinate, it suffices to show that $a \rest c \in \mu$ for every
  $c \in \mu$. Fix $c \in \mu$. Since $a \rest b \in \mu$ and $\mu$ is
  closed under $\cdot$, we have $(a \rest b) \cdot c \in \mu$. Since
  \[(a \rest b) \cdot c \just {\eqref{eq:14}} =a \rest (b \cdot c)
    \just{\eqref{left}} \leq a \rest c,\]
  it follows that $a \rest c \in \mu$ as required.

  Conversely, suppose $(\xi, \mu) \in a^\theta \rest b^\theta$. To show that $(\xi, \mu) \in (a \rest b)^\theta$, we must show that $a \rest b \in \mu$. Now $(\xi, \mu) \in a^\theta \rest b^\theta$ means $b \in \mu$ and there exists a maximal filter~$\nu$ such that
  $\nu \approx \xi$ and $a \in \nu$. In particular, since $\nu \approx
  \xi \approx \mu$ and $a \in \nu$, by \eqref{eq:15}, we have
  $[a] \in (\mu/{\sim})^\uparrow$. On the other hand, since $\mu$ is
  maximal, if it does not contain $a \rest b$, then it contains some
  $c$ satisfying $(a \rest b) \cdot c = 0$, and since $b, c \in \mu$,
  we have $[b \cdot c] \in \mu/{\sim}$. Thus, the filter
  $(\mu/{\sim})^\uparrow$ contains the element $[a] \wedge [b \cdot
  c]$. But using~\eqref{eq:2} and~\eqref{eq:14} in this order, we may
  compute
  \[[a] \wedge [b \cdot c]= [a \rest (b \cdot c)] = [(a \rest b) \cdot
    c] = [0].\]
  By Lemma~\ref{l:3}\ref{item:7}, this contradicts properness
  of~$\mu$.
\end{proof}

Notice that, as shown by the next example, the map $\theta$ of
Proposition~\ref{p:4} is not, in general, a representation of~$\cA$,
as it may fail to be injective.

\begin{example}
  Let $X = \{x, y, z\}$, and for a subset $S \subseteq X $ denote by
  ${\rm Id}_S$ the identity partial function on~$X$ with domain
  $S$. We let $\cA$ be the $\{\cdot, \rest\}$-algebra of partial
  functions with universe $\{{\rm Id}_\emptyset, {\rm Id}_{\{x\}},
  {\rm Id}_{\{x, y\}}, {\rm Id}_{\{x,z\}}\}$ (note that the operations
  $\cdot$ and $\rest$ coincide on~$\cA$). Then the unique maximal
  filter of~$\cA$ is $\{{\rm Id}_{\{x\}}, {\rm Id}_{\{x, y\}}, {\rm
    Id}_{\{x,z\}}\}$, and thus there is no maximal filter separating
  the elements ${\rm Id}_{\{x, y\}}$ and ${\rm Id}_{\{x,z\}}$. In
  particular, the map from Proposition~\ref{p:4} is not a
  representation of~$\cA$ by partial functions.
\end{example}

The rest of this section is devoted to showing that, if we replace
`restriction semilattice' by `difference--restriction algebra' in
the statement of Proposition~\ref{p:4}, the map $\theta$ becomes a
representation of~$\cA$ by partial functions. For that, we will use
some properties of maximal filters of subtraction algebras, and hence,
of difference--restriction algebras.

We now let $\cS$ be a subtraction algebra. We noted in
\Cref{preliminaries} that for every $a \in \cS$ we have a Boolean
algebra $a^\downarrow$. This will allow us to identify maximal filters of
$\cS$ with those of $a^\downarrow$. We recall that maximal filters of
Boolean algebras are also known as \emph{ultrafilters}, and they are
characterised as those filters $F$ such that for every element $b$ of
the Boolean algebra concerned, $b \in F \iff \overline{b} \notin F$,
where $\overline{b}$ denotes the complement of~$b$.

\begin{proposition}\label{l:4}
  For every $a \in \cS$, there is a bijection between
  ultrafilters of the Boolean algebra $a^{\downarrow}$ and maximal
  filters of $\cS$ containing $a$.

 More precisely: if $\mu \subseteq \cS$ is a maximal filter containing $a$, then
  $\mu \cap a^\downarrow$ is an ultrafilter of $a^\downarrow$, conversely if
  $\nu$ is an ultrafilter of $a^\downarrow$, then $\nu^\uparrow$
   (with the upward closure taken in $\cS$) is a
  maximal filter of $\cS$, and these constructions are mutually inverse.
\end{proposition}
\begin{proof}
  Suppose $\mu \subseteq \cS$ is a maximal filter, and let $a \in
  \mu$. It is easy to verify that $\mu \cap a^\downarrow$ is a proper
  filter of $a^\downarrow$.  It is also easy to verify that any filter
  $F$ of $a^\downarrow$ yields a filter $F^\uparrow$ of $\cS$ and that $(\mu \cap
  a^\downarrow)^\uparrow \subseteq \mu$. Hence
  any filter $F$ of $a^\downarrow$ properly extending $\mu \cap
  a^\downarrow$ satisfies $\mu \subsetneq F^\uparrow = \cS$, and hence $F =
  a^\downarrow$ (since $F$ is upward closed in $a^\downarrow$). That is, $\mu \cap
  a^\downarrow$ is an ultrafilter of~$a^\downarrow$.
  
  Conversely, let $\nu \subseteq a^{\downarrow}$ be an ultrafilter. It
  is straightforward to check that $\nu^\uparrow$ is a proper filter
  of~$\cS$, so we only need to show it is maximal. Let $F \subseteq
  \cS$ be a filter properly containing $\nu^\uparrow$ and let $b$
  belong to $F$ but not $\nu^\uparrow$.  We know both $a \bmeet b$ and $a - b$
  are in $a^\downarrow$ and are complements in this Boolean
  algebra. Hence either $a \bmeet b$ or $a - b$ is in the ultrafilter
  $\nu$, and hence in $\nu^\uparrow$. But $\nu^\uparrow$ is an upward-closed set
  that does not contain $b$, so it cannot contain $a\bmeet b$, and
  thus we have $a-b \in \nu^\uparrow \subseteq
  F$. Using~\eqref{disjoint_0}, this yields $b \bmeet (a-b) = 0 \in
  F$, so $F = \algebra A$. Hence $\nu^\uparrow$ is a maximal filter.

  Finally, we check that the two constructions are inverse to each
  other. If $\mu \subseteq \cS$ is a maximal filter and $a \in \mu$,
  then we have an inclusion of maximal filters $(\mu \cap
  a^\downarrow)^\uparrow \subseteq \mu$ and thus an equality. On the
  other hand, if $\nu \subseteq a^\downarrow$ is an ultrafilter, it is
  clear that $\nu = \nu^\uparrow \cap a^\downarrow$---this is
  the case for \emph{any} upward-closed subset $\nu$ of $a^\downarrow$.
\end{proof}

\begin{corollary}\label{max_condition}
  Let $F \subseteq \cS$ be a filter. Then the following are
  equivalent.
  \begin{enumerate}[label = (\alph*)]
  \item\label{equivalent_one} $F$ is maximal.
  \item\label{equivalent_two} For all $a \in F$ and $b \in \algebra
    S$, precisely one of $a \bmeet b$ and $a - b$ belongs to $F$.
  \item\label{equivalent_three} For some $a \in F$, for all $b \in
    \algebra S$, precisely one of $a \bmeet b$ and $a - b$ belongs to
    $F$.
  \end{enumerate}
\end{corollary}
\begin{proof}
  Since $F$ is nonempty, it is clear that \ref{equivalent_two} implies
  \ref{equivalent_three}, while \ref{equivalent_one} implying
  \ref{equivalent_two} is an immediate consequence
  of~\Cref{l:4}. Suppose $a \in F$ witnesses the truth of
  \ref{equivalent_three}. Then by this hypothesis, $\nu \coloneqq F
  \cap a^\downarrow$ is an ultrafilter of $a^\downarrow$. Therefore,
  by \Cref{l:4}, $\nu^\uparrow$ is a maximal filter of $\cS$, and
  clearly $\nu^\uparrow \subseteq F$. Since $F$ is proper (because, by
  hypothesis and that $a - 0 \in F$, we know $0 = a \cdot 0$ does not belong to $F$), we conclude
  that $F = \nu^\uparrow$, and hence $F$ is maximal.
\end{proof}

\begin{corollary}\label{extension}
  Let $F \subseteq \cS$ be a proper filter. Then there exists a
  maximal filter $\mu$ with $F \subseteq \mu$.
\end{corollary}
\begin{proof}
  Take an element $a \in F$. It is straightforward to check that $F
  \cap a^\downarrow$ is a filter of the Boolean algebra
  $a^\downarrow$. Let $\nu$ be an ultrafilter of $a^\downarrow$ that
  extends $F \cap a^\downarrow$. Then, by \Cref{l:4}, the set $\nu^\uparrow$
  is the required $\mu$.
\end{proof}





\begin{proposition}\label{representation}
  Let $\cA$ be a difference--restriction algebra, and let $\theta$ be
  the map given by
  \begin{equation*}
    a^\theta \coloneqq \{(\xi, \mu)\in \mf \cA \times \mf \cA \mid \xi
    \approx \mu \ \text{ and } \ a \in \mu \}.
  \end{equation*}
  Then $\theta$ is a representation of the $\{ -, \rest\}$-algebra
  $\algebra A$ by partial functions.
\end{proposition}
\begin{proof} By Proposition~\ref{p:4}, we already known that $\theta$
  is a map to partial functions on $\mf \cA$ and preserves  the $\rest$ operation. Let $\xi,
  \mu \in \mf \cA$ be $\approx$-equivalent and $a, b \in \cA$. Then
  \begin{align*}
    (\xi, \mu) \in a^\theta - b^\theta
    & \iff a \in \mu \ \text{ and }\ b \notin \mu
    \\  & \iff a-b \in \mu \qquad\qquad\qquad \text{(by \Cref{max_condition})}
    \\  & \iff (\xi, \mu) \in (a-b)^\theta.
  \end{align*}
  Therefore $\theta$ is a homomorphism of $\{-, \rest\}$-algebras.
  
  Finally, we show that $\theta$ is injective. Since $\theta$ is a
  homomorphism, we have $a^\theta = b^\theta$ if and only if
  $(a-b)^\theta = \emptyset = (b-a)^\theta$. In turn, by
  \Cref{extension} this holds exactly when $a-b = 0 = b-a$, which
  implies $a = b$.
\end{proof}

\begin{theorem}\label{axiomatisation}
  The class of $\{ -, \rest\}$-algebras representable by partial
  functions is a variety, axiomatised by the finite set of equations
  \ref{schein1} -- \ref{lifting}.
\end{theorem}

\begin{proof}
  As we saw in \Cref{preliminaries} all representable algebras validate the
  axioms. By \Cref{representation}, every $\{ -,
  \rest\}$-algebra validating the axioms is representable.
\end{proof}

We finish this section with an alternative representation of any representable $\{-,
\rest\}$-algebra, using only
\emph{injective} partial functions. This representation is built from the representation exhibited
in \Cref{representation}.

\begin{corollary}
  Let $\cA$ be difference--restriction algebra, and let $\eta$ be the
  map given by
  \begin{equation*}
    a^\eta \coloneqq \{(\dbracket \mu, \mu)\in (\mf \cA/{\approx})
    \times \mf \cA \mid  a \in \mu \}.
  \end{equation*}
  Then $\eta$ is a representation of~$\algebra A$ by injective partial
  functions.
\end{corollary}
\begin{proof}
  This is a simple consequence of \Cref{representation} together with
  the observation that, for all maximal filters $\xi, \mu \subseteq
  \cA$ and $a \in \cA$, we have
  \[(\xi, \mu)\in a^\theta \iff (\dbracket \xi, \mu) \in a^\eta. \qedhere\]
\end{proof}

\begin{corollary}\label{cor:rep}
  The class of $\{ -, \rest\}$-algebras representable by injective
  partial functions is a variety, axiomatised by the finite set of
  equations \ref{schein1} -- \ref{lifting}.
\end{corollary}
\section{Complete representability}\label{sec:complete}

In this section we discuss complete representations of $\{-,
\rest\}$-algebras and investigate the axiomatisability of the class of
completely representable algebras.

The next two definitions may apply to any function from a poset $\cP$
to a poset $\cQ$. So in particular, these definitions apply to
representations of Boolean algebras as fields of sets and to
representations of subtraction algebras or difference--restriction algebras as
algebras of partial functions, where a representation of a Boolean
algebra is viewed as an embedding into a full powerset algebra
$\mathcal P (X)$, and a representation of a
subtraction algebra or difference--restriction algebra is viewed as an embedding
into the algebra of all partial functions~$\mathcal P\mathcal F(X)$ on
some set~$X$. Since these are the only cases we are concerned with in
this section, and since existing meets/joins in both $\mathcal P(X)$
and $\mathcal P\mathcal F(X)$ are given by intersections/unions, we
will represent meets and joins in~$\cP$ by~$\meet$ and~$\join$
respectively and meets and joins in~$\cQ$ by~$\bigcap$ and~$\bigcup$
respectively, using subscripts if it is necessary to be more precise
about which poset we are in.

\begin{definition}\label{def:meet}
  A function $h: \cP \to \cQ$ is \defn{meet complete} if, for every
  nonempty subset~$S$ of~$\algebra{P}$, if $\meet S$ exists, then so
  does $\bigcap h[S]$ and
  \[h(\meet S) = \bigcap h[S]\text{.}\]
\end{definition}

\begin{definition}\label{def:join}
  A function $h: \cP \to \cQ$ is \defn{join complete} if, for every
  subset~$S$ of~$\algebra{P}$, if $\join S$ exists, then so does
  $\bigcup h[S]$ and
  \[h(\join S) = \bigcup h[S].\]
\end{definition}

Note that $S$ is required to be nonempty in \Cref{def:meet}, but not
in \Cref{def:join}. Despite the asymmetry, this is the natural choice
if we wish to formulate a definition of meet complete for partial
function algebras. Since the set $\mathcal P \mathcal F(X)$ has a top
element with respect to inclusion if and only if $X$ is a singleton,
requiring preservation of tops would prevent any (cardinality greater
than $2$) partial function algebra with a top from being meet
completely representable, including all finite ones and all completely
representable Boolean algebras (interpreted as $\{\rest, -\}$-algebras
of identity functions). Thus this would obstruct both meet complete
representability being an infinitary specialisation of
representability and partial function algebras being generalisations
of set algebras.


On the other hand, if a
representation of a Boolean algebra as a field of sets preserves the
existing nonempty meets, then it also preserves the top element (in a
Boolean algebra, we have $1 = a \vee \neg a$ for every $a$, and both
$\vee$ and $\neg$ are preserved by Boolean algebra homomorphisms). Thus for any algebra with a Boolean reduct our definition is not in conflict with the more usual definition. Readers who are uncomfortable with the deviation from standard terminology may choose to view our usage of \emph{meet complete} as shorthand for \emph{nonempty-meet complete}.

The clearest way to understand the underlying cause of the join--meet asymmetry is to realise that, for our purposes, the concepts that \emph{join} and \emph{meet} are providing formalisations of are `abstract union' and `abstract intersection'. Since the empty intersection does not (in an absolute sense) exist, the meet of the empty set will never have relevance for us.

A Boolean algebra can be represented using a meet-complete
representation if and only if it can be represented using a
join-complete representation, for the simple reason that any
meet-complete homomorphism \emph{is} join-complete, and vice versa. So
in this case we may simply describe such a homomorphism using the
adjective \defn{complete}.\begin{NoHyper}\footnote{In the case of
    representations of Boolean algebras as fields of sets, other
    adjectives have been used. Dana Scott suggested \emph{strong},
    which was subsequently used by Roger Lyndon; John Harding uses
    \emph{regular}.}\end{NoHyper} We will now see that the same
remarks apply to subtraction algebras, and hence to
difference--restriction algebras.

We now start to follow a part of \cite{complete} very closely---the end
of Section 2 and beginning of Section 3 there. The upcoming several
proofs (up to \Cref{atomistic}) are trivial adaptations of the proofs found in that paper, but
it is worth including them here, since they are all rather short. Note that although, in view of the subject of this paper, we choose to state some of these results in terms of representations of $\{-, \rest\}$-algebras by partial functions, the $\rest$ operation plays no role---the results hold more generally for representations of $\{-\}$-algebras by sets.

\begin{lemma}\label{l:15}
  Let $\algebra{A}$ and $\cB$ be subtraction algebras and $h: \cA
  \to \cB$ a homomorphism. For each $a\in \cA$, let $h_a: a^\downarrow
  \to h(a)^\downarrow$ denote the homomorphism of Boolean algebras
  induced by~$h$ (recall \Cref{c:3}). If $h$ is meet complete or
  join complete, then $h_a$ is complete.
\end{lemma}

\begin{proof}
  We will show that if $h$ is meet (respectively, join) complete then
  each $h_a$ is meet (respectively, join) complete. Since
   meet complete and join complete are equivalent notions for Boolean algebras, it
  follows in both cases that $h_a$ is complete, as required.
  
  Suppose $h$ is meet complete. If $S$ is a nonempty subset of
  $a^\downarrow$, then all lower bounds for $S$ in $\algebra{A}$ are
  also in $a^\downarrow$. Hence if $\meet_{a^\downarrow} S$ exists
  then it equals $\meet_{\algebra{A}} S$, and so $\bigcap_{\algebra B}
  h[S]$ exists and equals $h(\meet_{a^\downarrow} S)$. This equality
  also tells us that $\bigcap_{\algebra B} h[S] \in
  h(a)^\downarrow$. Hence $h(\meet_{a^\downarrow} S) =
  \bigcap_{\algebra B} h[S] = \bigcap_{h(a)^\downarrow} h[S]$. So
  ~$h_a$ is complete.

  Suppose that $h$ is join complete, $S \subseteq a^\downarrow$, and 
  $\join_{a^\downarrow} S$ exists. If $c \in \algebra{A}$ and $c$ is
  an upper bound for $S$, then $c \geq c \bmeet a \geq
  \join_{a^\downarrow} S$. Hence $\join_{a^\downarrow} S =
  \join_{\algebra{A}} S$, giving the existence of $\bigcup_{\algebra B} h[S]$ and
  the equality $h(\join_{a^\downarrow} S) = h(\join_{\algebra{A}} S) =
  \bigcup_{\algebra B} h[S]$. This equality also tells us that $\bigcup_{\algebra B} h[S] \in h(a)^\downarrow$. Hence $h(\join_{a^\downarrow} S) =\bigcup_{\algebra B} h[S] = \bigcup_{h(a)^\downarrow} h[S]$. So $h_a$ is complete.
\end{proof}

\begin{corollary}\label{lemma:Boolean}
  Let $\algebra{A}$ be an algebra of the signature $\{ -, \rest\}$. Any
  representation $\theta$ of $\algebra{A}$ by partial functions
  restricts to a representation of $a^\downarrow$ as a field of sets
  over $\theta(a)$, which is complete if $\theta$ is meet complete or
  join complete.
\end{corollary}

\begin{corollary}\label{cor:1}
  Let $\algebra{A}$ and $\cB$ be subtraction algebras and $h: \cA
  \to \cB$ be a homomorphism. If $h$ is meet complete, then it is join
  complete.
\end{corollary}
\begin{proof}
  Suppose that $h$ is meet complete. Let $S$ be a subset of
  $\algebra{A}$ and suppose that $\join_{\algebra{A}} S$ exists. Let
  $a = \join_{\algebra{A}} S$. Then $h_a$ is complete and so
  \[h(\join_{\algebra{A}} S) = h(\join_{a^\downarrow} S) = \bigcup_{h(a)^\downarrow}
    h[S]= \bigcup_{\algebra B}
    h[S]\text{.}\qedhere\]
\end{proof}

\begin{corollary}\label{cor:2}
  Let $\algebra{A}$ and $\cB$ be subtraction algebras and $h: \cA
  \to \cB$ be a homomorphism. If $h$ is join complete, then it is meet
  complete.
\end{corollary}
\begin{proof}
  Suppose that $h$ is join complete. Let $S$ be a nonempty subset of
  $\algebra{A}$ and suppose that $\meet_{\algebra{A}} S$ exists. As
  $S$ is nonempty, we can find $s \in S$. We let $S \bmeet s$ denote the set $\{s' \bmeet s \mid s' \in S\}$. Then $h_s$ is complete and
  \[h(\meet_{\algebra{A}} S) = h(\meet_{\algebra{A}} (S \bmeet s))
    = h(\meet_{s^\downarrow} (S \bmeet s)) = \bigcap_{h(s)^\downarrow} h[S \bmeet
    s] =\bigcap_{\algebra B} h[S \bmeet
    s] = \bigcap_{\algebra B} h[S]\text{.}\qedhere\]
\end{proof}

We have established that there is but one notion of complete homomorphism for representable $\{ -, \rest\}$-algebras. Hence there is but one notion of
complete representation for $\{ -, \rest\}$-algebras. If a $\{ -, \rest\}$-algebra has a complete representation we say it is \defn{completely representable}. 

We now move
on and consider the property of being atomic, both for algebras and
for representations. We will see that the completely representable
algebras are precisely the algebras that are representable and atomic.

\begin{definition}
Let $\algebra{P}$ be a poset with a least element, $0$. An \defn{atom} of $\algebra{P}$ is a minimal nonzero element of $\algebra{P}$. We write $\At(\algebra{P})$ for the set of atoms of $\algebra P$. We say that $\algebra{P}$ is \defn{atomic} if every nonzero element is greater than or equal to an atom. 
\end{definition}

We note that representations of $\{ -, \rest\}$-algebras necessarily
represent the partial order by set inclusion: this may be seen as a
consequence of \Cref{lemma:Boolean}. The following definition is
meaningful for any notion of representation where this is the case.

\begin{definition}
  Let $\algebra{P}$ be a poset with a least element and let $\theta$
  be a representation of $\algebra{P}$. Then $\theta$ is \defn{atomic}
  if $x \in \theta(a)$ for some $a \in \algebra{P}$ implies $x \in
  \theta(b)$ for some atom $b$ of $\algebra{P}$.
\end{definition}

We will need the following theorem.

\begin{theorem}[Hirsch and Hodkinson {\cite[Theorem 5]{journals/jsyml/HirschH97a}}]\label{thm:hirsch-hodkinson} Let $\algebra{B}$ be a Boolean algebra. A representation of $\algebra{B}$ as a field of sets is atomic if and only if it is complete.
\end{theorem}

\begin{proposition}\label{prop:at-com}
Let $\algebra{A}$ be an algebra of the signature $\{ -, \rest\}$ and $\theta$ be a representation of $\algebra{A}$ by partial functions. Then $\theta$ is atomic if and only if it is complete.
\end{proposition}

\begin{proof}
Suppose that $\theta$ is atomic, $S$ is a nonempty subset of $\algebra{A}$ and $\meet S$ exists. It is always true that $\theta(\meet S) \subseteq \bigcap \theta[S]$, regardless of whether or not $\theta$ is atomic. For the reverse inclusion, we have
\[\begin{array}{cll}
 & (x, y) \in \bigcap \theta[S]
\\ \implies & (x, y) \in \theta(s) & \text{for all }s \in S
\\ \implies & (x, y) \in \theta(a) & \text{for some atom }a\text{ such that }(\forall s \in S)\text{ } a \leq s
\\ \implies & (x, y) \in \theta(a) & \text{for some atom }a\text{ such that } a \leq \meet S
\\ \implies &  (x, y) \in \theta(\meet S)\text{.}
\end{array}\]
The third line follows from the second because, choosing an $s_0 \in S$ we have $(x, y) \in \theta(s_0)$, hence some atom $a$  with $(x, y) \in \theta(a)$, and thus $(x, y) \in \theta(a \bmeet s)$ for any $s \in S$. So for all $s \in S$, the element $a \bmeet s$ is nonzero, so equals $a$, by atomicity of $a$, giving $a \leq s$.

Conversely, suppose that $\theta$ is complete. Let $(x, y)$ be a pair contained in $\theta(a)$ for some $a \in \algebra{A}$. By \Cref{lemma:Boolean}, the map $\theta$ restricts to a complete representation of $a^\downarrow$ as a field of sets. Hence, by \Cref{thm:hirsch-hodkinson}, $(x, y) \in \theta(b)$ for some atom $b$ of the Boolean algebra $a^\downarrow$. Since an atom of $a^\downarrow$ is clearly an atom of $\algebra{A}$, the representation $\theta$ is atomic.
\end{proof}

\begin{corollary}\label{cor:atomic}
Let $\algebra{A}$ be an algebra of the signature $\{ -, \rest\}$. If $\algebra{A}$ is completely representable by partial functions then $\algebra{A}$ is atomic.
\end{corollary}

\begin{proof}
Let $a$ be a nonzero element of $\algebra{A}$. Let $\theta$ be any complete representation of $\algebra{A}$. Then $\emptyset = \theta(0) \neq \theta(a)$, so there exists $(x, y) \in \theta(a)$. By \Cref{prop:at-com}, the map $\theta$ is atomic, so $(x, y) \in \theta(b)$ for some atom $b$ in $\algebra{A}$. Then $(x, y) \in \theta(a \bmeet b)$, so $a \bmeet b > 0$, from which we may conclude that the atom $b$ satisfies $b \leq a$.
\end{proof}

For Boolean algebras, the algebra being atomic is necessary and sufficient for complete representability \cite{Abian01051971}. On the other hand, there exist scenarios in which being atomic is necessary but \emph{not} sufficient for complete representability, for example for the signature of composition, intersection, and antidomain, for representation by partial functions (\cite[Proposition 4.6]{complete}). Do we have sufficiency in our case? The answer is yes. But before we prove this we need a couple more lemmas.

\begin{definition}
A poset $\algebra{P}$ is \defn{atomistic} if its atoms are join dense in $\algebra{P}$. That is to say that every element of $\algebra{P}$ is the join of the atoms less than or equal to it.
\end{definition}

Clearly any atomistic poset is atomic. For subtraction algebras, and
in particular for $\{ -, \rest\}$-algebras representable by partial
functions, the converse is also true.

\begin{lemma}\label{atomistic}
Let $\algebra{A}$ be a subtraction algebra. If $\algebra{A}$ is atomic, then it is atomistic.
\end{lemma}

\begin{proof}
  Suppose $\algebra{A}$ is atomic and let $a \in \algebra{A}$. We know the algebra $a^\downarrow$ is a Boolean algebra and 
  clearly it is atomic. It is well-known that atomic Boolean algebras
  are atomistic. So we have
\[
a = \join_{a^\downarrow} \At(a^\downarrow)   = \join_{\algebra{A}}  \At(a^\downarrow) = \join_{\algebra{A}} \{ x \in \At(\algebra{A}) \mid x \leq a \}\text{.}
\]
The second equality holds because any upper bound $c \in \algebra{A}$ for $\At(a^\downarrow)$ is above an upper bound in $a^\downarrow$, for example $c \bmeet a$. Hence the least upper bound in $a^\downarrow$ is least in $\algebra{A}$ also.
\end{proof}

The last lemma concerns properties of the atoms of representable $\{-,
\rest\}$-algebras.

\begin{lemma}\label{l:atom_preserved}
  Let $\algebra{A}$ be a representable algebra of the signature $\{ -,
  \rest\}$. Then
  \begin{enumerate}[label=(\alph*)]
  \item\label{item:1} if $x \in \At(\cA)$, then $[x] \in \At
    (\cA/{\sim})$;
  \item\label{item:2} for every $a \in \cA$ and $x \in \At(\cA)$,
    either $x \rest a = 0$ or $x \rest a$ is an atom. And moreover, $x
    \rest a$ is an atom if and only if $x \preceq a$  (and if and only if  $x \sim x \rest a$).
  \end{enumerate}
\end{lemma}
\begin{proof}\leavevmode
  \begin{enumerate}[label = (\alph*)]
  \item Let $a \in \cA$ be such that $[a] \prec [x]$. Then $a \rest x
    < x$. (We cannot have $a \rest x = x$, else, by \Cref{l:2}, $[x]
    \preceq [a]$.)  But then as $x$ is an atom, $a \rest x = 0$. Hence
    $[a] = [a \rest x] = [0]$. As $[a]$ was an arbitrary element below
    $[x]$, we conclude that $[x]$ is an atom.

  \item Suppose $x \rest a \neq 0$ and let $b \in \cA$ be such that $0
    \leq b < x \rest a$. By \Cref{c:2}, we have $0 \preceq b \prec x
    \rest a$ and, by~\Cref{l:2}, $x \rest a \preceq x$. But by
    part~\ref{item:1}, $[x]$ is an atom, and thus, $0 \sim b$ and $x
    \rest a \sim x$. This yields $b = 0$, and so, $x \rest a$ is an
    atom, and $x \preceq a$. Finally, by~\Cref{l:2} we have $x \preceq
    a$ if and only if $x \sim x \rest a$. \qedhere
\end{enumerate}
\end{proof}

\begin{proposition}\label{complete_representation}
  Let $\algebra A$ be an atomic difference--restriction algebra. Let
  $\theta$ be the map given by
  \begin{equation*}
    a^\theta \coloneqq \{(x,y) \in \At(\algebra A) \times \At(\algebra
    A) \mid  \ x \sim y \ \text{ and }\ y \leq a\}.
  \end{equation*}
  Then $\theta$ is a complete representation of $\algebra A$ by
  partial functions.
\end{proposition}

\begin{proof}
  First we show that the relation $a^\theta$ is a partial
  function. That is, we argue that if $x \sim y \sim z$ for atoms $x,
  y, z$, with $y, z \le a$,  then $y = z$. But this is a consequence
  of Lemma~\ref{l:5}.

  Next we show that $\theta$ represents each operation correctly. We
  pick $a, b \in \algebra A$ and two
  $\sim$-equivalent atoms $x,y \in \At(\cA)$ .

  Showing that $(a-b)^\theta = a^\theta - b^\theta$ amounts to
  showing that
  \[y \leq a-b \iff (y \leq a\ \text{ and } \ y\not\leq b).\]
  By~\eqref{eq:10}, if $y \leq a-b$ then $y \leq a$. Suppose that
  we also have $y \leq b$. Then, $y \leq b\cdot (a-b)$ which,
  by~\eqref{disjoint_0}, yields $y = 0$, a contradiction. This shows
  the forward implication. Conversely, since $y \in
  \At(a^{\downarrow})$ and $a - b$ and $a \bmeet b$ are complements in the
  Boolean algebra $a^\downarrow$, it follows that $y \leq a - b$
  (because we are assuming $y \not\leq b$). Thus, we conclude that
  $\theta$ represents $-$ correctly.

  For $\rest$, first suppose that $(x,y) \in (a \rest
  b)^\theta$. By~\eqref{restricts}, we have $(x,y) \in b^\theta$. We
  show that $(x, y \rest a) \in a^\theta$, and thus $x \in \dom
  (a^\theta)$, yielding $(x, y) \in a^\theta \rest
  b^\theta$. By~\eqref{restricts}, we have $y \rest a \leq a$. Thus
  we only need to show that $x \sim y \rest a$. By \Cref{l:2}, we have
  $a \rest b \preceq a$, and since, by Lemma~\ref{l:3}, $\preceq$
  includes $\leq$, we have $y \preceq a \rest b$. Thus, $y \preceq a$
  and, again by~\Cref{l:2}, we have $y \sim y \rest a$. Since $x\sim
  y$ by hypothesis, we conclude $(x, y \rest a) \in a^\theta$ as
  claimed.

  Conversely, suppose $(x,y) \in a^\theta \rest b^\theta$, that is, $x
  \in \dom(a^\theta)$ and $y \leq b$. Since $y$ is an atom of the
  Boolean algebra $b^\downarrow$, we have $y \le a \rest b$ or $y \le
  b-(a \rest b)$. We suppose that $y \le b-(a \rest b)$ and we let $z
  \in \At(a^\downarrow)$ be $\sim$-equivalent to~$x$. Using that $\le$
  is included in $\preceq$, we have $y \preceq b- (a \rest b)$ and $z
  \preceq a$. Since $x \sim y \sim z$, it follows by~\Cref{l:2} that
  $x \preceq a \rest (b-(a \rest b))$. Now, using Lemma~\ref{l:2} and
  Corollary~\ref{c:1} in this order, we may compute:
  \[[a \rest (b - (a \rest b))] = [a] \wedge [b-(a\rest b)]] = [a]
    \wedge ([b] - [a]).\]
  Again by Corollary~\ref{c:1}, we known that $\cA/{\sim}$ is a
  subtraction algebra. Thus, by~\eqref{disjoint_0}, we have that $[a]
  \wedge ([b] - [a]) = [0]$, and by \Cref{l:3}\ref{item:7} it follows
  that $x = 0$, which is a contradiction. Therefore, we have $y \le (a
  \rest b)$ as intended.

  Next, we note that $\theta$ is injective. If $a^\theta = b^\theta$
  then $a$ and $b$ are greater than or equal to the same set of
  atoms. Since $\algebra A$ is atomistic (\Cref{atomistic}), $a$ and
  $b$ are each the supremum of this set of atoms, hence are equal.

   Finally, we
  show that $\theta$ is complete. By Proposition~\ref{prop:at-com},
  we know $\theta$ being complete is equivalent to it being atomic, and $\theta$
  is clearly atomic: for every $a \in \cA$ and $(x, y) \in a^\theta$,
  $y$ is an atom of $\cA$ such that $(x, y) \in y^\theta$.
\end{proof}

\begin{theorem}\label{complete_axiomatisation}
The class of $\{ -, \rest\}$-algebras that are completely representable by partial functions is axiomatised by the finite set of equations \ref{schein1} -- \ref{lifting} together with the $\forall\exists\forall$ first-order formula stating that the algebra is atomic.
\end{theorem}

\begin{proof}
By definition, the completely representable algebras are representable, and by \Cref{cor:atomic} they are atomic. \Cref{complete_representation} tells us the converse---that any representable and atomic algebra is completely representable. Hence the completely representable algebras are precisely those that are both representable and atomic. By \Cref{axiomatisation}, equations \ref{schein1} -- \ref{lifting} axiomatise representability. So with the addition of the formula stating the algebra is atomic, an axiomatisation of the completely representable algebras is obtained.
\end{proof}

As before, we can use the representation of \Cref{complete_representation} to
get a complete representation by \emph{injective} partial functions.
\begin{corollary}
  Let $\algebra A$ be an atomic difference--restriction algebra. Let
  $\eta$ be the map given by
  \begin{equation*}
    a^\eta \coloneqq \{([x],x) \in (\At(\algebra A)/{\sim}) \times \At(\algebra
    A) \mid  x \leq a\}.
  \end{equation*}
  Then $\eta$ is a complete representation of $\algebra A$ by
  injective partial functions.
\end{corollary}
\begin{proof}
  This follows from \Cref{complete_representation} together with the
  observation that for every $x, y \in \At(\cA)$ and $a\in \cA$, we
  have
  \[([y],x) \in a^\eta \iff (y, x) \in a^\theta. \qedhere\]
\end{proof}

\begin{corollary}\label{complete_corollary}
  The class of $\{ -, \rest\}$-algebras that are completely
  representable by injective partial functions is axiomatised by the
  finite set of equations \ref{schein1} -- \ref{lifting} together with
  the $\forall\exists\forall$ first-order formula stating that the
  algebra is atomic.
\end{corollary}

The axiomatisation of \Cref{complete_axiomatisation} and \Cref{complete_corollary} uses the minimum possible degree of quantifier alternation, for it is not possible to axiomatise these classes using any $\exists\forall\exists$ first-order theory, finite or otherwise.

\begin{proposition}
The class of $\{ -, \rest\}$-algebras that are completely representable by partial functions and the class of $\{ -, \rest\}$-algebras that are completely representable by injective partial functions are not axiomatisable by any $\exists\forall\exists$ first-order theory.
\end{proposition}

\begin{proof}
  Any Boolean algebra $\cB = (B, 0, 1, \wedge, \compl{\phantom{c}})$
  can interpret an algebra $\cB_{\{ -, \rest\}}$ of the signature $\{ -, \rest\}$ by
  setting $a - b \coloneqq a \wedge \compl{b}$ and $a \rest b
  \coloneqq a \wedge b$, and it is easy to check that $\cB$ equipped with these two operations satisfies axioms~\ref{schein1}--\ref{lifting}.
   Moreover, since the derived
  operation $\cdot$ is given by
  \[a \cdot b \just{\eqref{complement}}= a - (a - b) = a \wedge
    \overline{(a \wedge \overline b)} = a \wedge (\overline{a} \vee b)
    = a \wedge b,\]
  the orderings on $\cB$ and $\cB_{\{ -, \rest\}}$ coincide, and in particular $\cB_{\{ -, \rest\}}$ is atomic if and only if $\cB$ is. 
 On
  the other hand, there exist Boolean algebras $\algebra B$ and
  $\algebra B'$ with $\algebra B$ atomic and $\algebra B'$ not, such
  that $\algebra B$ and $\algebra B'$ satisfy the same
  $\exists\forall\exists$ first-order theory---see \cite[Proposition
  3.7]{complete} for a proof of this fact. Hence $\algebra B_{\{ -, \rest\}}$ and $\algebra B'_{\{ -, \rest\}}$ also have the same $\exists\forall\exists$ first-order theory as one another, since their basic operations are defined by terms in the Boolean signature. 
  Thus, by \Cref{complete_axiomatisation}/\Cref{complete_corollary},
  $\algebra B_{\{ -, \rest\}}$ and $\algebra B'_{\{ -, \rest\}}$ witness that any
  $\exists\forall\exists$ first-order theory cannot have all and only
  the completely representable algebras as its models.
\end{proof}

\section*{Declarations}
\subsection*{Acknowledgements}
The authors would like to thank the anonymous referees for the careful
reading of the paper and for their useful suggestions that helped to
improve the presentation of our work, in particular allowing us to present the results of Section~4
as a consequence of those from Section~3, thereby making the paper clearer.

\subsection*{Funding}
The first author was partially supported by the Centre for Mathematics
of the University of Coimbra - UIDB/00324/2020, funded by the
Portuguese Government through FCT/MCTES and partially supported by
the European Research Council (ERC) under the European Union's Horizon
2020 research and innovation program (grant agreement No. 670624).
The second author was partially supported by
the European Research Council (ERC) under the European Union's Horizon
2020 research and innovation program (grant agreement No. 670624) and partially supported by the Research Foundation -- Flanders (FWO) under the SNSF--FWO Lead Agency Grant 200021L 196176 (SNSF)/G0E2121N (FWO).
\subsection*{Conflicts of interest} The authors have no relevant financial or non-financial interests to disclose.
\subsection*{Data availability statement} Data sharing is not applicable to this article as no datasets were generated or analysed during the current study.

\bibliographystyle{amsplain}

\bibliography{../brettbib}
\end{document}

%% file: preamble.tex
\usepackage[british]{babel}
\usepackage{xr}
\usepackage{cleveref}
\usepackage[shortlabels]{enumitem}
\usepackage{mathtools}
\usepackage{bm}
\usepackage{centernot}
\usepackage{slashed}

\usepackage{microtype}
\usepackage{anyfontsize}

\DeclareMathOperator{\dom}{dom}

\DeclareMathOperator{\At}{At}

\newcommand \pref {\mathbin{\sqcup}}
\newcommand \bmeet {\cdot}

\newcommand \limplies {\mathbin{\rightarrow}}

\newcommand \rest {\mathbin{\vartriangleright}}
\newcommand \aand \wedge

\newcommand{\ppreceq}{\mathbin{\preceq\hspace{-1mm}\preceq}}
\newcommand{\lleq}{\mathbin{\leq\hspace{-1mm}\leq}}

\newcommand{\meet}{\prod}
\newcommand{\join}{\sum}

\newcommand{\compl}[1]{\overline{#1}}

\newcommand{\from}{\colon}

\newcommand{\algebra}[1]{\mathfrak{#1}}
\newcommand{\defn}[1]{\textbf{#1}}

\makeatletter
\newcommand\mynobreakpar{\par\nobreak\@afterheading}
\makeatother

\theoremstyle{cupthm}
\newtheorem{theorem}{Theorem}[section]
\newtheorem{proposition}[theorem]{Proposition}
\newtheorem{corollary}[theorem]{Corollary}
\newtheorem{lemma}[theorem]{Lemma}

\theoremstyle{cupdefn}
\newtheorem{definition}[theorem]{Definition}

\theoremstyle{cuprem}

\newtheorem{example}[theorem]{Example}

\numberwithin{equation}{section}
\usepackage{tikz}
\usetikzlibrary{arrows}
\usepackage{bookmark}
\usepackage{setspace}\onehalfspace
\newcommand{\cA}{\algebra A} \newcommand{\cB}{\algebra B}
 
\newcommand{\cP}{\algebra P} \newcommand{\cQ}{\algebra Q}
 
\newcommand{\cS}{\algebra S} \newcommand{\cT}{\algebra T}
 
\newcommand{\mf}[1]{{\sf Filt}_{\max}(#1)}
\newcommand{\just}[2]{\stackrel{#1}{#2}}

\newcommand{\filt}[1]{{\sf Filt}(#1)}
\newcommand{\dbracket}[1]{\llbracket #1 \rrbracket}